\documentclass[11pt,a4paper,leqno,twoside, headinclude]{amsart}

\usepackage[tracking=true]{microtype}
\DeclareMicrotypeSet*[tracking]{my}%
  { font = */*/*/sc/* }%
\SetTracking{ encoding = *, shape = sc }{ 45 }

\title[Positive curvature and rational ellipticity]{Positive curvature and rational ellipticity}
\def\titl{Positive curvature and rational ellipticity}
\def\auth{Manuel Amann and Lee Kennard}
\date{March 1st, 2014}


\author{\auth}
\thanks{The first author was supported by a research grant of the German Research Foundation. The second author was partially supported by National Science Foundation Grant DMS-1045292.}

\usepackage[english]{babel}

\usepackage[colorlinks,pdftex, plainpages=false]{hyperref}
\hypersetup{pdftitle=\titl, pdfauthor=\auth, pdftoolbar=false,
plainpages=false, hyperindex=true, pdfdisplaydoctitle=true}

\hypersetup{colorlinks,%
citecolor=black,%
filecolor=black,%
linkcolor=black,%
urlcolor=black,%
pdftex}

\usepackage{color}

\usepackage{amsmath, amssymb, amscd, amsthm}
\usepackage{stmaryrd}
\usepackage{latexsym}
\usepackage[all]{xy}
\usepackage{pb-diagram}
\usepackage{rotating}
\usepackage{multicol}
\usepackage{lscape}
\usepackage{wasysym}
\usepackage{longtable}
\usepackage{enumerate}

\xyoption{all}

\newtheorem{theo}{Theorem}[section]
\newtheorem{main}{Theorem}
\newtheorem{maincor}[main]{Corollary}
\newtheorem*{main*}{Theorem}

\newtheorem*{mainprop*}{Proposition}

\newtheorem{mainconj}{Conjecture}

\newtheorem{prop}[theo]{Proposition}
\newtheorem{defi2}[theo]{Definition}
\newtheorem*{defi2*}{Definition}

\newenvironment{defi*}{\begin{defi2*}\normalfont}{\end{defi2*}}

\newenvironment{defin*}[1]{\begin{defi2*}[#1]\normalfont}{\end{defi2*}}
\newtheorem{rem2}[theo]{Remark}
\newenvironment{rem}{\begin{rem2}\normalfont}{\hfill$\boxbox$\end{rem2}}

\newtheorem{lemma}[theo]{Lemma}
\newtheorem{cor}[theo]{Corollary}
\newtheorem*{cor*}{Corollary}

\newtheorem{conj}[theo]{Conjecture}
\newtheorem{prob}[theo]{Problem}
\newenvironment{conjec}[1]{\begin{conj}[#1]\normalfont}{\end{conj}}
\newtheorem{ques}[theo]{Question}
\newtheorem*{conj*}{Conjecture}
\newtheorem*{theo*}{Theorem}
\newtheorem*{ques*}{Question}
\newtheorem*{mi2}{Main Idea}

\newtheorem{ex2}[theo]{Example}

\newenvironment{exa}[1]{\begin{ex2}[#1]\normalfont}{\hfill$\boxbox$\end{ex2}}
\newtheorem{exer2}[theo]{Exercise}

\newtheorem{alg2}[theo]{Algorithm}

\newcommand{\cc}{{\mathbb{C}}}                                     
\newcommand{\hh}{{\mathbb{H}}}                                     
\newcommand{\qq}{{\mathbb{Q}}}                                     
\newcommand{\rr}{{\mathbb{R}}}                                     
\newcommand{\pp}{{\mathbf{P}}}                                     
\newcommand{\s}{{\mathbb{S}}}                                      
\newcommand{\SO}{{\mathbf{SO}}}                                    
\newcommand{\U}{{\mathbf{U}}}                                      
\newcommand{\SU}{{\mathbf{SU}}}                                    
\newcommand{\Sp}{{\mathbf{Sp}}}                                    
\newcommand{\E}{{\mathbf{E}}}                                      
\newcommand{\F}{{\mathbf{F}}}                                      
\newcommand{\Spin}{{\mathbf{Spin}}}                                
\newcommand{\dif} {{\operatorname{d}}}                             
\newcommand{\In} {{\,\subseteq\,}}                                 
\newcommand{\Isom}{{\operatorname{Isom}}}                          
\newcommand{\cat} {{\operatorname{cat}}}                           
\newcommand{\rk}{{\operatorname{rk\,}}}                            
\newcommand{\co}{\colon\thinspace}                                 

\newcommand{\comment}[1]{}                                         
\newcommand{\xto}[1]{\xrightarrow{#1}}                             
\newcommand{\hto}[1]{\overset{#1}{\hookrightarrow}}                

\newcommand{\biq}[2]{#1\;\!\!\!\sslash \;\!\!\!#2}                 
\newcommand{\case}[1]{\textbf{Case #1.}}                           
\newcommand{\str}{\noindent\textbf{Structure of the article. }}    
\newcommand{\odd}{\textrm{odd}}                                    
\newcommand{\even}{\textrm{even}}                                  

\newcommand{\Z}{\mathbb{Z}}
\newcommand{\Q}{\mathbb{Q}}

\newcommand{\embedded}{\hookrightarrow}
\DeclareMathOperator{\cod}{cod}

\newcommand{\floor}[1]{\left\lfloor #1 \right\rfloor}
\newcommand{\pfrac}[2]{\left(\frac{#1}{#2}\right)}
\newcommand{\of}[1]{\left(#1\right)}

\DeclareMathOperator{\dk}{dk}

\newenvironment{prf}{\begin{proof}[\textsc{Proof}]} {\end{proof}}     

\begin{document}

 \thispagestyle{empty}


\begin{abstract}
Simply-connected manifolds of positive sectional curvature $M$ are speculated to have a rigid topological structure. In particular, they are conjectured to be rationally elliptic, i.e., all but finitely many homotopy groups are conjectured to be finite.

In this article we combine positive curvature with rational ellipticity to obtain several topological properties of the underlying manifold (see \cite{GroveHalperin82}). These results include a small upper bound on the Euler characteristic and confirmations of famous conjectures by Hopf and Halperin under additional torus symmetry. We prove several cases (including all known even-dimensional examples of positively curved manifolds) of a conjecture by Wilhelm.

\end{abstract}

\maketitle


\section*{Introduction}

The question of whether a manifold admits a positively curved Riemannian metric has been addressed in various ways over the decades. Astoundingly, all this effort did not bring forth a long list of examples. All the known simply-connected examples---mainly homogeneous spaces and biquotients---have one striking topological property in common. They are what is called \emph{rationally elliptic spaces}, meaning that their total rational homotopy $\pi_*(\cdot)\otimes \qq$ is finite dimensional. (See \cite{Ziller07} for a survey of examples, see Dearricott \cite{Dearricott11}, Grove--Verdiani--Ziller \cite{Grove-Verdiani-Ziller11}, and Petersen--Wilhelm \cite{Petersen-Wilhelm09pre} for two new examples in dimension seven and see \cite{FHT01} for an compendium on rational homotopy theory.) Rational ellipticity on its own has interesting consequences. In particular, it implies the upper bound on the sum of the Betti numbers $2^{\dim M}$ conjectured by Gromov.

This motivates a crucial conjecture due to Bott which states that an (almost) non-negatively curved manifold is rationally elliptic (see \cite{GroveHalperin82}). A conjecture by Hopf states that a positively curved manifold has positive Euler characteristic, which in the elliptic case endows the rational cohomology algebra with nice features; in particular, odd Betti numbers vanish. According to one of the most famous conjectures in rational homotopy theory made by Halperin, these \emph{positively elliptic} or \emph{$F_0$-spaces} $F$ have the property that any fibration of simply-connected spaces $F\hto{} E\to B$ satisfies that $H^*(E;\qq)\cong H^*(F;\qq)\otimes H^*(B;\qq)$ as modules---see Section \ref{sec01}.

In the 1990s, Grove propagated the idea to focus on positively curved metrics that admit large isometry groups. For several results in this article we shall follow this approach in the rationally elliptic setting. The measure of symmetry we will consider in this paper is the \textit{symmetry rank}, which is the rank of the isometry group. In other words the symmetry rank of a Riemannian manifold is at least $r$ if and only if there exists an effective, isometric action of a torus of dimension $r$.

Many topological classification results of varying strengths (diffeomorphism, homeomorphism, homotopy, etc.) have been proven under the assumption that the symmetry rank is sufficiently large. The classification theorems of Grove--Searle (see \cite{GroveSearle94}) and Wilking (see \cite{Wilking03}) are prototypical examples. For related results and context, we refer the reader to the comprehensive surveys of Grove \cite{Grove09}, Wilking \cite{Wilking07} and Ziller \cite{Ziller??}.

In this article
we shall derive strong restrictions on the Euler characteristic from both above and below in the rationally elliptic case with torus symmetry. Moreover, in this setting we shall prove the classical Halperin conjecture from rational homotopy theory, and we apply rational techniques to prove a conjecture of Fred Wilhelm on the relation between the dimensions of base and total space within a submersion of positively curved manifolds. Finally, we characterize those biquotients---biquotients form one of the largest classes of known rationally elliptic spaces---which might admit a metric of positive curvature and logarithmic symmetry. We refer the reader to Section \ref{sec01} for a detailed outline of the concepts which could merely be sketched here.

We begin with the announced statement on the Wilhelm conjecture. Since we can formulate it with (rational) fibrations of topological spaces instead of Riemannian submersions of manifolds, it certainly provides a stronger statement in the depicted cases. In particular, it obviously implies the Wilhelm conjecture for manifolds that are currently known to admit positive sectional curvature.
\begin{main}\label{theoE}
Let $F\hto{}M^n\to B$ be a non-trivial fibration of simply-connected compact manifolds. Suppose that $M$ satisfies one of the following:
\begin{itemize}
\item
Either $H^*(M;\qq)$ is singly generated as an algebra, or
\item
$M$ is rationally one of the known even-dimensional examples of manifolds carrying positive curvature, or
\item
$M$ is rationally a hermitian symmetric space, and $M$ is positively curved with symmetry rank at least $2\log_{2} n+6$.
\end{itemize}
Then it holds that $\dim B>\dim F$.
\end{main}
Note that in the last item in the theorem, we do not require that there is a positively curved hermitian metric on $M$. This theorem is proved in Theorems \ref{theo05}, \ref{theo06}, \ref{theo03}. Compare \cite[Theorem 1, p.~202]{Bro62} for a similar result on spheres. See \cite{Su14} for a $32$-manifold the rational cohomology of which is singly generated, although it is not a projective space.

Recall that the known even-dimensional simply-connected positively curved examples which do not have singly generated rational cohomology are the flag manifolds in dimensions $6$, $12$, $24$ and the Eschenburg biquotient $\biq{\SU(3)}{T^2}$.

The second conjecture we want to address is the famous conjecture by Halperin stating that given a fibration with fiber being rationally elliptic of positive Euler characteristic, then the Leray--Serre spectral sequence of the fibration degenerates at the $E_2$-term. Several reformulations of this statement are known.

Let us state a corollary of a more technical statement.
\begin{maincor}[positively curved Halperin conjecture with symmetry] \label{corC}
Let $M^n$ be a rationally elliptic closed manifold with positive sectional curvature and symmetry rank greater than
$2\log_2 n + n/8$.
Then $M$ satisfies the Halperin conjecture.
\end{maincor}
We remark that due to Wilking (see \cite[Theorem 5]{Wil03}, \cite[Lemma 4.2, p.~16]{AK13}) one sees directly that the same result holds under $n/6+2\log_2(n)-3$ symmetry rank. (Indeed, in this case the rational cohomology algebra is generated by at most two elements---cf.~Theorem \ref{theo07}.)

Our main tool to prove this result is the following theorem. Indeed, it effectively restricts the number of generators of the cohomology algebra.
Theorem \ref{theoD} improves some of the main theorems in \cite{AK13} in the special case of rationally elliptic manifolds.
\begin{main}\label{theoD}
Let $M^{n}$ be rationally elliptic of positive curvature and with symmetry rank at least
$2\log_2 n+n/(2k)$
for $k\geq 2$. We obtain that
\begin{align*}
0\leq \chi(M)\leq 2^{k-2} ( n/k + 1)
\end{align*}
and, if $\chi(M)>0$ the number of generators $l$ of the rational cohomology algebra $H^*(M;\qq)$ is restricted by
$1\leq l \leq k$.
\end{main}
In \cite{AK13} we prove that the Euler characteristic is bounded by a constant multiple of $n \log n$ without the assumption of rational ellipticity. Here we obtain a linear bound, which matches the growth rate of $\cc\pp^n$.

Note that in the light of the next theorem, Theorem \ref{theoC}, the condition on positive Euler characteristic is virtually usually satisfied; it is always satisfied for large $n$ (with ``large'' depending on $k$). In this case the structure theory of rationally elliptic spaces yields that odd rational cohomology groups vanish and $\chi(M)=\sum \dim H^*(M)$ so that we get effective bounds on the size of the entire rational cohomology algebra.
\begin{main}[elliptic Hopf conjecture with symmetry]\label{theoC}
Let $M^{2n}$ be a simply-connected rationally elliptic manifold admitting a metric of positive curvature with symmetry rank at least $\log_{4/3} (2n)$. Then the Euler characteristic of $M$ satisfies $\chi(M)\geq 2$.
\end{main}

Due to \cite{Ken13} the Hopf conjecture is known under logarithmic symmetry given that the dimension of the manifold is divisible by $4$. In \cite{AK13} we proved the Hopf conjecture for arbitrary dimensions, logarithmic symmetry rank and under the assumption that one of the Betti numbers $b_2$, $b_3$ or $b_4$ vanishes. The previous theorem is a general version for all dimensions of rationally elliptic spaces with logarithmic symmetry.

\vspace{3mm}

\textbf{As a general convention all cohomology algebras are taken with rational coefficients (unless stated differently). All cochain algebras are over the rationals. All manifolds are closed.}

\vspace{3mm}

\str In Section 1 we provide necessary background knowledge and we illustrate the problems under consideration. The following sections are devoted to the proofs of the respective theorems and corollaries. In Section \ref{sec05} we prove Theorem \ref{theoD} and Corollary \ref{corC}; in Section \ref{sec8} we give the proof of Theorem \ref{theoC}. The next section, Section \ref{sec02}, is devoted to formulating several problems on the structure of positively curved manifolds and to reconciling them with the conjecture by Wilhelm. We then prove Theorem \ref{theoE}.

\section{Preliminaries}\label{sec01}

As mentioned in the introduction, a prominent conjecture in the field of nonnegative curvature is
\begin{conj*}[Bott]
An (almost) non-negatively curved manifold is rationally elliptic.
\end{conj*}

Recall that a simply-connected topological space $X$ is called \emph{rationally elliptic}, if $\dim \pi_*(X)\otimes \qq$ is finite and $\dim H^*(X,\qq)$ is finite. In other words, the first condition means that from some degree on all the rational homotopy groups of $X$ vanish. Obviously, the second condition is trivial on manifolds.

Rationally elliptic spaces satisfy several relations on the degrees of their rational homotopy groups, which we shall make use of in the course of this paper. See \cite[p.~434]{FHT01} for a summary of these.

There is another classical conjecture which combines nicely with the Bott conjecture.
\begin{conj*}[Hopf]
A positively curved manifold of even dimension has positive Euler characteristic.
\end{conj*}
Recall that a space $X$ is \emph{positively elliptic} or $F_0$, if it is rationally elliptic and has positive Euler-characteristic. These positively elliptic spaces have very strict properties; in particular, their odd Betti numbers vanish.

The Bott conjecture is trivial for homogeneous spaces (respectively biquotients) and it is known for manifolds of cohomogeneity one and two (see \cite{GH87} and \cite{Yea12}). For some confirmations of the Hopf conjecture under symmetry see \cite{Ken13}, \cite{AK13} and references therein.

The Bott and Hopf conjectures combine nicely to speculating that an even dimensional manifold of positive curvature is an $F_0$-space. For $F_0$-spaces there is the \emph{Halperin conjecture}, which is one of the central conjectures in Rational Homotopy Theory.
\begin{conj*}[Halperin]
Let $F\hto{j} E \to B$ be a fibration of simply-connected spaces with $F$ being $F_0$. Then the fibration is \emph{rationally totally non-\linebreak[4]cohomologous to zero}, i.e.~the induced homomorphism $j^*\co H^*(E;\qq)\to H^*(F;\qq)$ is surjective.
\end{conj*}
Equivalently, the rational Leray-Serre spectral sequence degenerates at the
$E_2$-term or the cohomology module of the total space splits as the product of fiber and base cohomologies.

Several attempts have been made to confirm this
conjecture and various special cases have been established
(cf.~\cite{FHT01}). Just a few of them: The conjecture holds true,
if the cohomology algebra $H^*(F;\qq)$ has at most $3$ generators
(cf.~\cite{Tho81}, \cite{Lup90}) (see below), if all the generators are of the
same degree (cf.~\cite{SZ75}, \cite{SZ75a}), in the ``generic case''
(cf.~\cite{PP96}), on homogeneous spaces (cf.~\cite{ST87}) or if the
manifold $M^{2n}$ admits a Hard-Lefschetz structure
\begin{align*}
H^{n-k}(M,\rr) &\xto{\cong} H^{n+k}(M,\rr)\\
x&\mapsto \omega^k\cdot x
\end{align*}
(for a $2$-form $\omega$) such as K\"ahler manifolds do. This last
result was established/recalled in the article
\cite{Mei83}---cf.~\cite{Bla56} for the original statement. It
relies on the article \cite{Mei82} by Meier where he gives a very
nice reformulation---which we shall draw on---of the Halperin conjecture in terms of
self-homotopy equivalences of $F$ or via negative-degree derivations
on the rational cohomology algebra.

The importance of the conjecture surpasses the theory-internal
interest. So, for example, the same problem appears in the
deformation of singularities (see \cite{Wah83}) or in the existence question of nonnegatively curved metrics (see \cite{BK03}).

We shall need the following reformulation of the Halperin conjecture (see \cite[Theorem A, p.~329]{Mei82}).
\begin{theo}[Meier]\label{theo08}
An $F_0$-space satisfies the Halperin conjecture if and only if its
rational cohomology algebra does not permit non-trivial derivations
of negative degree.
\end{theo}
From this characterization it follows easily that once the cohomology algebra is generated by one element only, it satisfies the Halperin conjecture. In fact, using this characterization the subsequent Theorem was proved by Lupton in \cite{Lup90}.
\begin{theo}\label{theo07}
The Halperin conjecture holds for an $F_0$-space $M$, if the rational cohomology algebra of $M$ has at most $3$ generators.
\end{theo}

\vspace{3mm}

For an outline of the Wilhelm conjecture see Section \ref{sec02}.

\section{Proof of Theorem \ref{theoD} and of Corollary \ref{corC}}\label{sec05}

We prove Theorem \ref{theoD} and derive Corollary \ref{corC} on the Halperin conjecture as an easy consequence.

Recall in both statements that $M^n$ is a simply connected, closed manifold with even dimension. Assume $M$ admits a Riemannian metric with positive sectional curvature invariant under a torus action of rank at least $\frac{n}{2k} + 2\log_2 n$.

Recall that we wish to show that $\chi(M) \leq 2^{k-2}\left(\frac{n}{k}+1\right)$ and that, if $\chi(M) > 0$, then the number $l$ of generators of $H^*(M;\qq)$ satisfies $l\leq k$ with equality only if the minimal model of $M$ contains the minimal model of $\s^2$ as a sub-differential graded algebra.

\begin{proof}[\textsc{Proof of Theorem \ref{theoD}}]
We first prove the second assertion of the theorem. That is, we assume $\chi(M) > 0$, and we show that the number of generators $l$ of the cohomology algebra $H^*(M;\qq)$ is bounded above by $k$, with equality only if the minimal model $(\Lambda V, \dif)$ for $M$ fits into a rational fibration
	\[(\Lambda\langle e,e'\rangle, \dif e' = e^2) \embedded (\Lambda V, \dif) \to (\Lambda W, \overline{\dif}).\]
over an $F_0$-algebra $(\Lambda W,\bar \dif)$ (due to \cite[Theorem 28.6, p.~375]{FHT01}).

Since $M$ is an $F_0$ space, there exists a pure model \linebreak[4]$(\Lambda\langle x_1,\ldots,x_l,y'_1,\ldots,y'_l\rangle, \dif)$ where $\dif x_i = 0$, $\dif y_i' \in \Lambda\langle x_1,\ldots,x_l\rangle$, the $x_i$ have even degree, and the $y_i'$ have odd degree. Set $y_i = \dif y_i'$. We may reindex such that $\deg x_1 \leq \ldots \leq \deg x_l$ and $\deg y_i \geq 2 \deg x_i$ for all $i$.  For this we point the reader to Remark \ref{xrem}.

Set $c = \floor{\frac{n+4}{k}}$. Observe that the symmetry rank is at least $2 \log_2 n + \frac{c}{2} - 1$. By \cite{Kennard-pre2}, we conclude that $H^*(M)$ is $4$-periodic until degree $c$. In fact, \cite{Kennard-pre2} implies that $b_{c} = b_{c-4} \leq 1$. This implies that in these low degrees the truncated rational cohomology algebra $H^{*\leq c}(\cdot;\Q)$ of $M$ is isomorphic to that of $\cc\pp^\infty$, $\hh\pp^\infty$ or of $\s^2\times \hh\pp^\infty$. In the notation above, we have one of the following four cases:
	\begin{enumerate}
	\item (Sphere case) All $\deg y_i - \deg x_i \geq \deg x_i >c$.
	\item ($\cc\pp$ case) $\deg x_1 = 2$, $\deg y_1 > c $, and all other $\deg y_i - \deg x_i >c$.
	\item ($\hh\pp$ case) $\deg x_1 = 4$, $\deg y_1 >c$, and all other $\deg y_i - \deg x_i >c$.
	\item ($\s^2 \times \hh\pp$ case) $\deg x_1 = 2$, $\deg y_1 = 3$, $\deg x_2 = 4$, $\deg y_2 >c$, and all other $\deg y_i - \deg x_i >c$.
	\end{enumerate}
In the first case, the dimension formula implies
	\[l\pfrac{n+4}{k} < l(c+1) \leq \sum \deg y_i - \sum \deg x_i = n,\]
so $l \leq k-1$. In the second and third cases, one similarly has
	\[(c+1) - 4 + (l-1)(c+1) \leq \deg y_1 - \deg x_1 + \sum_{i=2}^l (\deg y_i - \deg x_i) = n,\]
so
	\[l \leq \frac{n+4}{c+1} < k.\]
In the final case, one has
	\[[4-2] + [(c+1) - 4] + (l-2)(c+1) \leq \sum \deg y_i - \sum \deg x_i = n,\]
so $l \leq k$. As a quick check, we note that these estimates on $l$ can be realized if we do not require $M$ to satisfy the curvature condition. For example, an $n$--dimensional product of $k-1$ spheres of dimensions at least $c + 1 = \floor{\frac{n+4}{k}} + 1$ realizes the bound $l = k -1$ in the first case. Replacing the first spherical factor by a $\cc\pp^m$, $\hh\pp^m$, or $\s^2 \times \hh\pp^m$ of the appropriate dimension shows that the estimate is sharp in these cases as well.

Since $l = k$ only occurs in the $\s^2 \times \hh\pp$ case, and since $(\Lambda\langle x_1,y_1\rangle, \dif)$ is a sub-dga of the minimal model of $M$, this concludes the proof of the first assertion of the theorem.

We proceed to prove the assertion on the Euler characteristic. Due to Theorem \ref{theoC} we obtain that $M$ has positive Euler characteristic if $\log_{4/3} n \leq \frac{n}{2k} + 2\log_2 n$. In any case the Euler charateristic of an elliptic space is always non-negative and it is zero if and only if $b_\odd\neq 0$---which is clearly the case if $M$ is odd-dimensional. Thus for proving the theorem we may assume that $\chi(M)>0$ and $b_\odd=0$, i.e.~$M$ is an $F_0$-space---see \cite[Proposition 32.10, p.~444]{FHT01}. We adopt the notation from above for a pure model $(\Lambda\langle x_1,\ldots,x_l,y_1',\ldots,y_l'\rangle,\dif)$ for $M$. Again write $\dif y_i' = y_i$, and recall that the $x_i$ are of even degree with $\dif x_i = 0$, and the $y_i$ are also of even degree with $\deg y_i \geq 2\deg x_i$.

The Euler characteristic $\chi(M)=\chi(H(\Lambda V,\dif))$ does not depend on the differentials in the minimal model, only on the generators $x_i$ and $y_i'$.
This follows from the formula
\begin{align*}
\chi(M)=\prod_{1\leq i\leq l} \deg y_i/\deg x_i
\end{align*}
from \cite[Proposition 32.15.(iii), p.~448]{FHT01}.

We consider first Case (1), where $M$ is $\floor{n/k}$ connected. The solution to the following optimization problem gives an upper bound for the Euler characteristic:

\begin{align*}
\operatorname{maximize} && \prod_{1\leq i\leq l} \deg y_i/\deg x_i \\
\operatorname{over}		&& l\geq 1, \deg x_i, \deg y_i\\
\operatorname{subject} \operatorname{to}  &&\sum_{1\leq i\leq l} \deg y_i-\deg x_i=n\\
\operatorname{and }&& \deg x_i > n/k, \deg y_i\geq 2 \deg x_i
\end{align*}

We claim that the solution to this problem is $2^{k-1}$. Indeed, this value can be achieved by taking $l = k-1$, $\deg y_i = 2\deg x_i$, and choosing appropriate values of $\deg x_i > n/k$. Topologically, this corresponds to a product of spheres.

To prove that $2^{k-1}$ is the maximum, we regard $\deg x_i$ and $\deg y_i$ are real numbers and prove in this case that the maximum is at most $2^{k-1}$. First, the maximum is achieved along the subset with $\deg y_i = 2 \deg x_i$. Indeed, if some $\deg y_i = 2\deg x_i + \delta$ for some $\delta > 0$, one can replace $\deg y_i$ by $\deg y_i - \frac{\delta}{3}$ and $\deg x_i$ by $\deg x_i + \frac{\delta}{3}$. Second, the product $\prod \deg y_i/\deg x_i = 2^l$, so it suffices to show that $l = k-1$. Indeed, $l \geq k$ is impossible by the constraints, but $l = k - 1$ can be arranged. This concludes the proof in the case where $M$ is $c$--connected.

The proofs in the other three cases are similar. The extremal case is the last case. Here one obtains that $\s^2 \times \hh\pp^m \times \s^{m_1}\times\cdots\times\s^{m_{k-2}}$ achieves the maximum possible value for $\chi(M)$. Subjecting $m$ and the $m_i$ to the constraints, one concludes that $m+1 \leq \frac{n}{2k} - \frac{1}{2}$, and hence that
	\[\chi(M) \leq 2^{k-1}(m+1) \leq 2^{k-2}\left(\frac{n}{k} + 1\right).\]
\end{proof}

\begin{rem}\label{xrem}
For the convenience of the reader we shall now give a simple argument for the result from the proof of \cite[Theorem 32.6, p.~443]{FHT01} that we may order the $x_i$ and the $y_i$ in such a way that
\begin{align}\label{eqn16}
2 \deg x_i\leq \deg y_i
\end{align}
Indeed, we may argue as follows: The $y_i$ form a regular sequence. Denote by $I(\cdot)$ the ideal generated by specified elements in $\qq[x_1,\dots,x_l]$.

We may assume that the $x_i$ are ordered by degree, beginning with the smallest. Now we choose a permutation of the $y_i$ according to the following rule: We choose $y_{i_l}$ such that $y_{i_l}$ does not lie in the ideal $I(x_1,\dots, x_{l-1})$. We now choose $y_{i_{l-1}}$---out of the remaining $y_i$---such that $y_{i_{l-1}}$ does not lie in $I(x_1,\dots, x_{l-2})$, etc., i.e.~we always have $y_{i_{j}} \not \in I(x_1,\dots, x_{j-1})$.

Such a choice can always be made, since the $y_i$ (and every permutation of them) form a regular sequence. Indeed, assume the contrary. That is, we suppose that at some point there is no valid choice of $y_j$. This means that all the remaining $(k-j)$-many $y_i$ lie completely in the ideal $I(x_1,\dots, x_{k-j-1})$. This implies that there is an $\tilde l$ with $l-j\leq \tilde l \leq l$ such that no power of $[x_{\tilde l}]$ vanishes in cohomology; a contradiction to the finite-dimensionality of $H^*(X)$.

Consequently, each $y_{i_j}$ has a nontrivial summand which is formed of factors out of the $x_{{j}},\dots,x_{{l}}$ only---the wordlength in these factors is at least $2$. Since the $x_i$ were ordered by degree, we have that $\deg x_j\leq \deg x_{j+1}\leq \ldots\leq \deg x_{l}$. This implies that the summand, and thus the entire term $y_{i_j}$ satisfies
$2\deg x_j\leq \deg y_{i_j}$.
This proves that the number of algebra generators $l$ of $H^*(M)$ satisfies
\begin{align}\label{eqn11}
\chi(M)\geq 2^l
\end{align}
provided $\chi(M)>0$.
\end{rem}

Assuming this theorem, we reduce Corollary \ref{corC} to previously known results.

\begin{proof}[Proof of Corollary \ref{corC}]
Recall that $M^n$ has symmetry rank at least $\frac{n}{8} + 2\log_2 n$. Taking $k = 4$ in Theorem \ref{theoD}, we conclude either that the cohomology of $M$ has at most three generators or that it has four generators and fits into a rational fibration
	\[(\Lambda\langle e,e'\rangle, \dif e' = e^2) \embedded (\Lambda V, \dif) \to (\Lambda W, \overline{\dif}).\]
In the first case, $M$ satisfies the Halperin conjecture by Theorem \ref{theo07} in the case of at most three generators.

In the second case, the $H(\Lambda W,\overline{\dif})$ is also $F_0$ (due to \cite[Theorem 28.6, p.~375]{FHT01}) and has at most three generators. By Lupton's theorem again, it satisfies the Halperin conjecture. Due to \cite[Theorem 1, p.~154]{Mar90}, we conclude that $M$ satisfies the Halperin conjecture.
\end{proof}

\vspace{1in}

\begin{rem}
We remark that for a fixed $k$ (not depending on $n$) the upper bound lies in $\mathcal{O}(n)$ and is optimal in this sense. Assume $k=4$, i.e.~a symmetry rank of $2\log_2 n+n/8$ at least, then we obtain
\begin{align*}
0\leq \chi(M)\leq n+4
\end{align*}
\end{rem}

\begin{cor}\label{cor02}
If $M^n$, $n$ even, is Hard-Lefschetz, positively elliptic, positively curved and has symmetry rank at least $2\log_2 n+n/8$, then $M$ is either a rational $\cc\pp^{n/2}$ or
has the rational homotopy type given  by a minimal model of the form
\begin{align*}
(\Lambda \langle u,a,x,y\rangle,\dif) \qquad & \deg u=2, \deg a\geq n/4+1,\\& \deg y=n+1-\deg a\neq 2\deg a-1=\deg x,\\&  \dif u=\dif a=0 \\&
\dif x=a^2+k_1 au^{\deg a/2}+ k_2 u^{\deg a}\\&
\dif y=au^{(n-2\deg a+2)/2}+k_3u^{(n-\deg a+2)/2} \textrm{ for } k_i\in \qq
\end{align*}
\end{cor}
\begin{prf}
Due to Theorem \ref{theoD} we know that there are at most $4$ generators of the rational cohomology algebra of $M$. The proof of the theorem tells us that, since $M$ looks like $\cc\pp^\infty$ in low degrees, the number of generators is at most $3$.

We have to show that there cannot be exactly three algebra generators of $H^*(M)$.
Suppose $(\Lambda \langle u,a,b,x,y,z\rangle,\dif)$ with $\deg u=2$, $\deg a\geq n/4+1$, $\deg b\geq n/4+1$ is the minimal model of $M$. Since $M$ is Hard-Lefschetz, the dual of an element lies in the ideal generated by $[u]$; in particular, so do the elements $[a^2], [b^2], [ab]$. Thus, due to Hard-Lefschetz, we conclude that $\deg x=\deg a^2-1=2\deg a-1$, $\deg y=\deg b^2-1=2\deg b-1$, $\deg z=\deg ab-1=\deg a+\deg b-1$. Using this information we compute the dimension of $M$ in terms of the degrees and we obtain the contradiction
\begin{align*}
n=\dim M&=(2\deg a+2\deg b+\deg a+\deg b)-(2+\deg a +\deg b)\\&=2 \deg a+2\deg b-2
\\&\geq 2\cdot (n/2+2)-2\\&>n
\end{align*}

Thus there may be at most one algebra generator of $H^*(M)$ above degree $n/4$ and the minimal model of $M$ has the form
\begin{align*}
(\Lambda \langle u,a,x,y\rangle,\dif)\qquad & \deg u=2, \deg a\geq n/4+1, \dif u=\dif a=0
\end{align*}
Due to Hard-Lefschetz we derive $\deg x=2 \deg a -1$ and
\begin{align*}
\dif x=a^2+k_1 au^{\deg a/2}+ k_2 u^{\deg a}.
\end{align*}
Consequently, we have $\deg y=n+1-\deg a$. Thus we have to consider two cases: In the first case $\deg x=\deg y=(2n+1)/3$, $\deg a=\frac{n+2}{3}$ implying that $\langle [a^2],[au^{(n+2)/6}], [u^{(n+2)/3}]\rangle\In H^{(2n+4)/3}(M)$ is just $1$-dimensional; a contradiction to Poincar\'e duality. In the second case we derive
\begin{align*}
\dif y=au^{(n-2\deg a+2)/2}+k_3u^{(n-\deg a+2)/2}
\end{align*}
due to Hard-Lefschetz again.
\end{prf}
Note that the given homotopy type includes the one of $\SO(n+2)/\SO(n)\times \U(1)$. We leave it to the reader to adapt this result
to quaternionically Hard-Lefschetz manifolds. Depending on the degrees of the elements in the minimal model, further restrictions on the $k_i$ apply. Clearly, in the case of a cohomologically symplectic $M$ (not necessarily Hard-Lefschetz) the number of generators of the cohomology algebra still can be $3$ at most.

\begin{rem}
Let us remark on a very special case of the theorem. If $M^{2n}$ is a positively curved manifold with the rational homotopy type of a simply-connected compact symmetric space $N$. If $M$ has symmetry rank at least $2 \log_2 2n + (2n)/8$, then $N$ is a product of a $\cc\pp^k$, $\hh\pp^k$, $\hh\pp^k\times \s^2$, $\s^k$, $\SO(k+2)/\SO(k)\times \SO(2)$, $\SO(k+3)/\SO(k)\times \SO(3)$ with at most two further spherical factors.

This follows directly from \cite[Corollary D]{AK13} and the fact that $M$ has to be highly periodic.
\end{rem}

Note that on a hermitian symmetric space we need less symmetry.
\begin{prop}\label{prop01}
If $M^{2n}$ is a positively curved manifold with the rational homotopy type of a simply-connected hermitian symmetric space $N$ and if $M$ has symmetry rank at least $2\log_2 (2n)+6$, then $N$ is either $\cc\pp^{n}$ or $\SO(n+2)/\U(1)\times \SO(n)$.
\end{prop}
\begin{prf}
We make use of the classification of irreducible hermitian symmetric spaces. Such a space is one of
\begin{align}\label{eqn18}
\nonumber&M_1:=\SU(p+q)/\mathbf{S}(\U(p)\times \U(q)),\\
\nonumber&M_2:=\SO(2n)/\U(n),\\
 &M_3:=\Sp(n)/\U(n),\\
\nonumber &M_4:=\SO(n+2)/\SO(n)\times \SO(2), \\
\nonumber &M_5:= \E_6/\SO(10)\times \SO(2),\\
\nonumber &M_6:=  \E_7/\E_6\times \SO(2)
\end{align}
In the case of $M_1$ we may assume that $p,q\geq 2$, since $M_1$ is complex projective otherwise. For $M_2$ we may assume that $n>3$ by the classification results in small dimensions.
We compute the respective first few Betti numbers as
\begin{align*}
M_1\co & b_2=1, b_4=2\\
M_2\co & b_2=1, b_4=1, b_6=2\\
M_3\co & b_2=1, b_4=1, b_6=2\\
M_5\co & b_2=1, b_4=1, b_6=1, b_8=2\\
M_6\co & b_2=1, b_4=1, b_6=1, b_8=1, b_{10}=2\\
\end{align*}
In the last two cases this can be deduced using the degrees of the non-trivial rational homotopy groups---each one-dimensional---of the exceptional space $\E_6$ respectively $\E_7$ as being $3,9,11,15,17,23$ and $3,11,15,19,23,27,35$.

The symmetry assumptions guarantees $4$-periodicity up to degree $14$, in particular, in our case, the Betti numbers have to satisfy $b_2=b_4=b_6=b_8=b_{10}=1$. Thus we deduce that the only hermitian symmetric spaces which resemble $\cc\pp^\infty$ in degrees $2$ to $10$ are the ones from the assertion. Note that no Euclidean products may arise, since they have $b_2>1$.(Clearly, for small $n$ only $\cc\pp^n$ may arise.)
\end{prf}

\vspace{3mm}
Let us end this section by making a speculation: Let $(\Lambda V,\dif)$ be a simply-connected elliptic minimal Sullivan algebra. Denote by $V'$ its spherical cohomology, i.e.~the subspace of $V$ with vanishing differential.
Then does $H(\Lambda V,\dif)\geq 2^{\dim V'}$ hold (with equality if and only if the algebra is isomorphic to the product of $\dim V'$-many spheres)?

This inequality in the case of (simply-connected) $F_0$-spaces is classical and the content of \eqref{eqn11}. Clearly, in this case the cohomology algebra $H^*(X)$ has exactly $(\dim V)/2$-many algebra generators which exactly correspond to the space $V'$.

\section{Proof of Theorem \ref{theoC}}\label{sec8}

In this section we shall prove the Hopf conjecture under logarithmic symmetry rank.
\begin{lemma}\label{lemma01}
Let $M^{n}$, $n\geq 8$ even, be a simply-connected rationally elliptic space with four-periodic rational cohomology. Then $\chi(M)>0$ and $H^\odd(M)=0$.
\end{lemma}
\begin{prf}
By definition (see \cite{Ken13}), it suffices to show that $b_3(M)=0$. First, assume $n\equiv 0\bmod{4}$. Due to four--periodicity and Poincar\'e duality we conclude $b_3(M) = b_{n-1}(M) = 0$.

Now assume that $n\equiv 2\bmod{4}$. Moreover assume that $M$ is not a rational sphere. It follows from four--periodicity and Poincar\'e duality that $b_2(M) = b_4(M) = 1$ and
\begin{align}\label{eqn17}
\chi(M) = 2 + \frac{n - 2}{4}\of{2 - b_3(M)}.
\end{align}
Since $M$ is rationally elliptic, and since $n \geq 10$, this inequality implies $b_3(N) \leq 3$. However, $b_3(N)$ must be even by Poincar\'e duality, four-periodicity and the graded commutativity of the cup product, hence $b_3(M) \leq 2$. In this case, however, we compute that $\chi(M) > 0$, i.e.~$H^\odd(M)=0$.
\end{prf}
\begin{rem}
For the sake of completeness we discuss the Euler characteristic of a six-dimensional simply-connected rationally elliptic space. (Here four-periodicity is losing its meaning.)

Due to Poincar\'e duality, the intersection form is non-degenerate and skew symmetric and $b_3$ is even. Inequality \eqref{eqn17} specializes to $0\leq 4-2b_3(M)$ and this holds if and only if $b_3(M)\in \{0,2\}$. Due to Poincar\'e duality the Euler characteristic of $M$ is given as $\chi(M)=2+2b_2(M)-b_3(M)$. Thus if $b_3(M)=0$, we compute $\chi(M)>0$. If $b_3(M)=2$, then ellipticity yields that $\chi(M)=0$. This can only be satisfied if $b_1(M)=b_2(M)=b_4(M)=b_5(M)=0$. Trivially, in this case $M\simeq_\qq \s^3\times \s^3$.
\end{rem}

\begin{proof}[\textsc{Proof of Theorem \ref{theoC}}]
It suffices to prove that $\chi(M)>0$. Indeed, a rationally elliptic space $M$ of positive Euler characteristic has vanishing odd-degree Betti numbers---see \cite[Proposition 32.10, p.~444]{FHT01}. Poincar\'e duality then implies $\chi(M)\geq 2$.

We mimic the proof of \cite[Theorem A]{AK13}. Recall our notation from there. Let $T$ be a torus acting effectively on $M$. If $N\In M$ is a submanifold on which $T$ acts, we set $\operatorname{dk}(N):= \dim \ker(T|_N)$.

We proceed by induction over the dimension. If $n\leq 48$, the symmetry assumption is larger than $n/4+1$ (and larger than $n/2$ for $n<10$), hence the result follows from the homotopy classification (resp. the diffeomorphism classification) of Wilking \cite{Wilking03} (resp. Grove--Searle \cite{GroveSearle94}). Therefore let $n \geq 50$.

As in \cite{AK13}, we have to differentiate two cases. Consider the isotropy representation of the isometrically acting torus at a fixed-point $x$.

\begin{description}
\item[Case 1] There exists $x \in M^T$ such that every $\iota\in\Z_2^{\rk T}$ with $\dk\of{M^\iota_x} \leq 1$ and $\cod\of{M^\iota_x} \leq \frac{n}{2}$ actually has $\cod\of{M^\iota_x} \leq \frac{n}{4}$.
\item[Case 2] For all $x \in M^T$, there exists $\iota\in \Z_2^{\rk T}$ with $\dk\of{M^\iota_x} \leq 1$ and $\frac{n}{4}<\cod\of{M^\iota_x} \leq \frac{n}{2}$.
\end{description}

In the first case we conclude in \cite[Theorem A]{AK13} that $M$ has $4$-periodic cohomology and Lemma \ref{lemma01} yields the result. Hence may assume that Case 2 applies.

Cover the torus fixed-point set by the respective $M_x^\iota$ from Case 2. We use the fact that the fixed-point set of a circle action on a rationally elliptic space is again rationally elliptic. We observe the following:
	\begin{itemize}
	\item If $\dk(M_x^\iota)=0$, we consider the effective $T$-action on $M_x^\iota$. There exists an $\s^1\In T$ together with a fixed-point component $F\In (M_x^\iota)^{\s^1}$ satisfying $x\in F$ and $\dk(F)=1$. Thus $F$ is rationally elliptic and has at least symmetry rank $\log_{4/3}(\dim F)$. The induction assumption applies to $F$ and $H^\odd(F)=0$.
	\item If $\dk(M^\iota_x) = 1$ and $M^\iota_x$ is a fixed-point component of a circle in $T$, then $M^\iota_x$ is rationally elliptic. The induction hypothesis applies to $M^\iota_x$, hence $H^\odd(M^\iota_x) = 0$.
	\item If $\dk(M^\iota_x) = 1$, $M^\iota_x$ is fixed by a circle $S^1$, but $M^\iota_x$ is strictly contained in $M^{S^1}_x$, then replace $\iota$ by the involution in $S^1$ and proceed as in one of the previous cases.
	\end{itemize}

In this manner we cover the whole fixed-point set $M^T$ with fixed-point components with vanishing odd Betti numbers. Since the sum of the odd Betti numbers of a such a component of $N^T$ is smaller equal the sum of the odd Betti numbers of $N$, we have $0=b_\odd(M^T)$. Since $\chi(M) = \chi(M^T) > 0$, we conclude from rational ellipticity that $H^\odd(M) = 0$.
\end{proof}
\begin{rem}
Suppose $M^n$ is a closed, one-connected, rationally elliptic manifold. If $M$ admits an effective action by a $T=\s^1$ such that $M^T$ has an isolated fixed point, then $\sum b_{2i+1}(M) = 0$ and hence $\chi(M) \geq 2$.
Indeed, assume that $M^T$ has an isolated fixed point and that $M$ is rationally elliptic. According to \cite[Corollary 7.49, p.~302]{FOT08} we derive that if the homotopy Euler characteristic satisfies $\chi_\pi(M)\neq 0$, then $M^T$ has no contractible component. Hence the assumed isolated fixed point of $M^T$ implies that $\chi_\pi(M)=0$. Since the rational ellipticity of $M$ implies $\chi(M)\geq 0$ with equality if and only if $\chi_\pi(M)\neq 0$ (see \cite[Proposition 32.10, p.~444]{FHT01}), we conclude $\chi(M) > 0$. This together with the assumption of rational ellipticity (see \cite[Proposition 32.10]{FHT01}) implies $\sum b_{2i+1}(M) = 0$, hence $\sum b_{2i+1}(M^T) = 0$ by Conner's theorem.
\end{rem}


\section{On a conjecture by Fred Wilhelm and the proof of Theorem \ref{theoE}} \label{sec02}

Clearly, it is desirable to try to fixate the rational homotopy type of an (elliptic) manifold of positive curvature and symmetry. However, at this stage every approach---even using equivariant rational homotopy theory---seems to be doomed to fail unless the following question---admittedly formulated in a suggestive way and vastly generalising the Hopf conjecture for $\s^2\times \s^2$---is answered at least in parts.
\begin{ques*}
Does positive curvature oppose the existence of rational product structures?
\end{ques*}
One positive answer to this question would show that the rational cohomology algebra cannot split into products.

The Halperin conjecture is clearly a conjecture on the structure of the cohomology algebra of a space. It seems not to differ between products and ``non-products'', since it holds in the somehow extremal cases of Cartesian products as well as Hard-Lefschetz manifolds.

Fred Wilhelm stated the following conjecture
\begin{conjec}{Wilhelm}\label{conj01}
Let $M\to B$ be a (non-trivial) Riemannian submersion with $M$ a complete positively curved manifold, then $2\dim B>\dim M$.
\end{conjec}
In the following we want to assume all positively curved manifolds to be compact, i.e.~due to the proven Soul conjecture, we assume $M$ not to be diffeomorphic to $\rr^n$. Then the submersion is a fibre bundle (with fiber $F$) according to Ehresmann and the conjecture can be stated as $\dim B>\dim F$.

This conjecture seems to be very hard to tackle. Nonetheless, in this section we shall even generalize this question to the following
\begin{ques}\label{ques01}
Let $M$ be a closed simply-connected manifold admitting a metric of positive curvature. If $M$ fits into a (non-trivial rational) fibration $F\hto{} M \to B$ of simply-connected manifolds (or topological spaces of finite type and finite formal dimension), when does $\dim B>\dim F$ hold?
\end{ques}
Note that this question cannot always be answered in the positive, since the Aloff-Wallach spaces rationally split as a product of $\s^2$ and $\s^5$.

We want to make the following conjecture and suggest it as yet another generalized form of the Hopf conjecture on $\s^2\times\s^2$. Assuming parts of it, we shall try to partially answer question \ref{ques01}. We shall justify this below.

\begin{prob}\label{conj02}
\begin{itemize}
Let $M^n$ be a simply-connected closed manifold admitting a metric of positive sectional curvature.
\item
If $n$ is even, then rationally $M$ does not split as a (non-trivial) product.
\item
If $n$ is odd, then rationally $M$ is a formal elliptic pure space with the dimension of the odd degree spherical rational homotopy equal to one, i.e.~the minimal model of $M$ splits as the product of a pure elliptic algebra of positive Euler characteristic and a factor which is the minimal model of an odd-dimensional sphere.
\end{itemize}

\vspace{3mm}

For large dimensions (at least larger than $24$) we may give several more speculations:
Let $M$ be a large-dimensional simply-connected closed manifold admitting a metric of positive curvature. Then rationally $M$ does not split as a (non-trivial) fibration, i.e.~whenever
\begin{align*}
\big(V,\dif ) \hto{}( \Lambda (V\oplus W),\dif) \to (W,\bar \dif)
\end{align*}
is a rational fibration (of Sullivan algebras) and $(\Lambda (V\oplus W),\dif)$ is a model for $M$, then
\begin{description}
\item[weak version]
$(\Lambda (V\oplus W),\dif)$ is not minimal as a Sullivan algebra.
\item[strong version] the rational transgression $\dif_0\co W^* \to V^{*+1}$ in the long exact homotopy sequence of the fibration is injective on $W^\odd$.
\item[alternate version] with $(x_i)$ a basis of $W^\odd$ and $(x_i')$ a basis of $V^\odd$, the following holds:
\begin{itemize}
\item[-] $\sum \deg x_i< \sum \deg x_i'$
\item[-] the projections $\dif(x_i)|_{\Lambda V}$ form a regular sequence in $\Lambda V^\even$ with the property that $\dif(\Lambda V)\cap \dif(W^\odd)|_V=0$.
\end{itemize}
\end{description}
\end{prob}
(For the rational transgression see \cite[p.~214]{FHT01}.)

No transition of positive curvature to topology seems to be known (from a rational point of view) which is as good as the transition of non-negative curvature to rationally elliptic spaces. The Problem tries to fill this gap by introducing reasonable structures which might be satisfied by positively curved manifolds.

\begin{rem}
Let us comment on this conjecture and its hierarchy. First, note that this conjecture clearly \emph{does not} imply that $M$ cannot be the total space of a fibration in any case.

The low-dimensional examples of positive curvature clearly do admit rational fibration structures. We shall give concrete depictions below.

A rationally elliptic space of positive Euler characteristic is formal. Thus a combined Bott--Hopf conjecture would yield the formality of even--dimensional simply-connected positively curved manifolds. Thus it is tempting to conjecture this for the odd dimensional examples, too. An elliptic pure algebra necessarily splits in the form given in the Problem, if it is formal.

Let us come to the high-dimensional version:
\begin{itemize}
\item The weak version already generalizes the classical Hopf conjecture on $\s^2\times\s^2$, i.e.~in particular, no product structures may appear in rational homotopy.
\item It seems easy to conjecture that large enough positively curved manifolds are homotopically simple, like rational CROSSes. We introduced the alternate version to show that a less restrictive, reasonable and more involved conjecture can be stated.
    In the next example we shall show that this version from a rational viewpoint is not good enough to deduce the Wilhelm conjecture.
\item
The strong version is a special case of the alternate one, where the regular sequence is given by the trivial one consisting of a subset of a basis of $W^\even$. The weak version does not a priori imply the strong version. For this consider the next example.
\end{itemize}
\end{rem}
\begin{exa}{alternate version does not imply Wilhelm conjecture}
The alternate version of Problem \ref{conj02} does not imply the Wilhelm conjecture purely from a rational viewpoint:

Consider the following rational fibration. In the notation from above, set $V:=\langle y_1,y_2,x'_1,x'_2\rangle$ with $\deg y_1=\deg y_2=4$, $\deg x_1'=\deg x_2'=39$, $\dif y_1=\dif y_2=0$, $\dif x_i'=y_i^{10}$. Set $W:=\langle x_1,x_2\rangle$ with $\deg x_1=27$, $\deg x_2=47$, $\dif x_1=y_1^7+y_2^7$, $\dif x_2=y_1^6y_2^6$, where $\dif$ is already the twisted differential in the rational fibration. We compute that $\deg x_1+\deg x_2=74<78=\deg x_1'+\deg x_2'$ and $\dim |(\Lambda W,\bar \dif)|=74>72=\dim |(\Lambda V,\dif)|$.
\end{exa}

\begin{exa}{weak version does not imply strong version}
Set $V:=\langle c,x'\rangle$ with $\deg c=4$, $\deg x'=7$, $\dif c=0$, $\dif x'=c^2$. Set $W:=\langle y_1,y_2,x_1,x_2\rangle$ with $\deg y_1=\deg y_2=2$, $\deg x_1=\deg x_2=3$, $\dif y_1=\dif y_2=0$, $\dif x_1=y_1^2+y_2^2$, $\dif x_2=y_1y_2+c$ where $\dif$ is already the twisted differential in the rational fibration. We compute a minimal model of the total space as $(\Lambda \langle y_1,y_2,x_1,x'\rangle,\dif)$ with $\dif y_i=0$, $\dif x_1=y_1^2+y_2^2$, $\dif x'=y_1^2y_2^2$. This minimal model does not split as a fibration in the sense of the weak conjecture. However, in the sense of the strong conjecture, the transgression $\dif_0$ in the fibration vanishes on $x_2$.

This shows that the weak version does not imply the strong one a priori.
\end{exa}

In contrast to the previous observations, we observe that the strong version is good enough to deduce the Wilhelm conjecture.
\begin{lemma}\label{lemma02}
Let $F\hto{} E\to B$ be a (non-trivial) fibration of simply-connected rationally elliptic spaces.
Suppose the strong conclusion of Problem \ref{conj02} holds for $E$ (not necessarily admitting positive curvature). Then it holds that $\dim F< \dim B$.
\end{lemma}
\begin{prf}
We compute the dimension of $B$ using the cohomological information of $F$. That is, let $x_1, \dots , x_n$ be a basis of $W^\odd$ where $(\Lambda W,\bar \dif)$ is the minimal model of $F$. It is trivial to see using the dimension formula (cf.~\cite[p.~434]{FHT01}) that
\begin{align}\label{eqn12}
\dim F=\dim |(\Lambda W,\bar\dif)|\leq \dim \prod_{1\leq i\leq n}\s^{\deg x_i}=\bigg(\sum_{1\leq i\leq n} \deg x_i\bigg).
\end{align}
According to Problem \ref{conj02}, the rational transgression on $W^\odd$ is injective, i.e.~there is an $n$-dimensional (not necessarily homogeneous) subspace $\langle y_1,\dots, y_n\rangle\In V^\even$ where $(\Lambda V,\dif)$ is the minimal model of $B$. Since $(\Lambda V,\dif)$ is rationally elliptic, it has finite dimension. Consequently, we obtain that $\dim V^\odd\geq n$. As in the proof of Theorem \ref{theoD} we may order a basis $(z_i)_{1\leq i\leq n'}$ of $V^\odd$ such that $\deg z_i\geq 2 \deg y_i$ for $1\leq i\leq n$. As a consequence, we derive that
\begin{align*}
\dim B=\dim |(\Lambda V,\dif)|\geq \prod_{1\leq i\leq n} \s^{\deg y_i}=\sum_{1\leq i\leq n} \deg y_i=\bigg(\sum_{1\leq i\leq n} \deg x_i\bigg)+n
\end{align*}
(For this estimate we use \cite[Theorem 32.6.ii, p.~441]{FHT01}, which gives the sum of all even-degree generators as a lower bound on the dimension.)

Comparing this to estimate \eqref{eqn12} proves the result.
\end{prf}
\begin{rem}
\begin{itemize}
\item The estimate $\dim B-\dim F\geq n$ from the lemma is sharp, as the example $\s^3\hto{} T_1\s^4\to \s^4$ shows. From \cite{Grove-Verdiani-Ziller11}, \cite{Dearricott11} we recall that the total space in this case does carry positive curvature.
\item
We remark that if $F$ is $F_0$ we may improve the estimate on its dimension by
\begin{align*}
\dim F\leq \dim \prod_{1\leq i\leq n}\cc\pp^{(\deg x_i-1)/2}=\bigg(\sum_{1\leq i\leq n} \deg x_i\bigg)-n.
\end{align*}
and the gap in dimension is $\dim B-\dim F\geq 2n$. Also this estimate is sharp due to $\cc\pp^1\hto{}\cc\pp^2\to \hh\pp^1$, the twistor submersion over the positive quaternion K\"ahler manifold $\hh\pp^1$ of positive curvature (with $\cc\pp^2$ carrying positive curvature).
\end{itemize}
\end{rem}
As we remarked we suggest Problem \ref{conj02} as a rational version of the Wilhelm conjecture. That is, transitioning positive curvature as rationally elliptic to the world of rational spaces, the strong version of Problem \ref{conj02} implies the Wilhelm conjecture.
\begin{theo}
We assume the Bott conjecture as well as the strong version of Problem \ref{conj02} to be true.
It follows that the Wilhelm conjecture holds true.
\end{theo}
\begin{prf}
If the total space $M^{n}$ of a Riemannian submersion is positively curved, so is the base space due to O'Neill. Due to the Bott conjecture, both spaces are rationally elliptic. The long exact sequence of the fibration tells us that also the fiber is rationally elliptic. Lemma \ref{lemma02} yields the result.
\end{prf}

\vspace{5mm}

After these warm-up considerations and speculations let us now head towards the proof of Theorem \ref{theoE}. For this we shall need the following recent result (see \cite[p.~10]{Yea12}, \cite{Hal}).
\begin{theo}[Halperin]\label{theo02}
Let $F\hto{} X\to B$ be a fibration where $X$ is a rationally elliptic nilpotent space. Suppose that $B$ is simply-connected, rationally of finite type, and that $\cat B<\infty$. If the Betti numbers $\dim H_i(F)$ grow at most polynomially in $i$, then $B$ is rationally elliptic.
\end{theo}

The proof of Theorem \ref{theoE} now splits into various separate theorems, which actually prove a stronger version. That is, they confirm Question \ref{ques01} in the respective cases.
\begin{theo}\label{theo05}
Let $M$ be simply-connected with singly generated rational cohomology algebra. If $M$ is the total space of a (rational) fibration of simply-connected spaces of finite type and finite formal dimension $F\hto{}M \to B$, then $\dim B>\dim F$.
\end{theo}
\begin{prf}
In this case $M$ is necessarily rationally elliptic.
Consider a fibration $F\hto{} M\xto{p} B$. We assume the fibration to be non-trivial.

Since the fiber cohomology is finite dimensional and since a finite CW-complex has finite Lusternik-Schirelmann category (cf.~\cite[Proposition 27.5, p.~354]{FHT01}), we may apply Theorem \ref{theo02} and the long exact homotopy sequence of a fibration to deduce that all three of $F$, $M$ and $B$ are rationally elliptic.

We now distinguish two cases. Either the rational cohomology of $M$ is concentrated in even degrees or $M$ is a rational odd-dimensional sphere.

\case{1} In this case the Euler characteristic of $M$ is positive. Since the Euler characteristic is multiplicative,
both $\chi(F)>0$ and $\chi(B)>0$. This implies that both $F$ and $B$ are $F_0$ spaces (with rational cohomology concentrated in even degrees). Consequently, the associated Leray--Serre spectral sequence degenerates at the $E_2$-term and
\begin{align}\label{eqn13}
H^*(M)\cong H^*(B)\otimes H^*(F)
\end{align}
as a module.

In the terminology from above the model of the fibration is of the form
\begin{align}\label{eqn14}
\big(\Lambda (V\oplus W),\dif\big) \cong \big(\Lambda V',\tilde \dif\big)\otimes (C,\tilde \dif_0)
\end{align}
with $(\Lambda V,\dif)$ a minimal model of $B$ and $(\Lambda W,\bar \dif)$ a minimal model of $F$ and with a contractible algebra $(C,\tilde \dif_0)$.
Due to observation \eqref{eqn13} we may choose a \emph{pure model} of the fibration---see \cite{Tho81}, \cite{Lup98}.

Thus, due to the properties of the splitting \eqref{eqn14}, we obtain that
\begin{align*}
\ker \bar \dif|_W \In
\ker \dif|_{V'}
=\langle u\rangle.
\end{align*}
Since $F$ is not trivial, it follows that $H^*(F)$ is generated by $[u]$ as an algebra, i.e.~we have $(\Lambda W,\bar\dif)=(\Lambda \langle u,u'\rangle,\bar \dif)$ with $\bar \dif u=0$, $\bar \dif u'=u^k$ for some $k>1$.

Since $H^*(M)$ is generated by one element, $[u]$, which also generates $H^*(F)$, we derive that $H^*(B)$ is also singly generated. Indeed, the only element from $\Lambda W$ which may lie in the contractible algebra $(C,\tilde \dif_0)$ is $u'$. Hence, there is at most one element $v\in V^\even$ in $C$, i.e.~$(C,\tilde \dif_0)\cong (\Lambda\langle u',v\rangle, (u'\mapsto v))$. Accordingly, this implies that $(\Lambda V,\dif)=(\Lambda \langle v,v'\rangle, (v' \mapsto v^l))$ for some $l>1$, since otherwise $H^*(M)$ could not be singly generated. Since $\tilde \dif_0 u'=v$ we compute $\deg u'=\deg v-1$ and $\dim F=\deg v-\deg u$, $\dim B=\deg v'+1-\deg v$ and
\begin{align*}
\dim B- \dim F=\deg v'+1-2\deg v+\deg u>\deg v'+1-2\deg v\geq 0
\end{align*}

\case{2} Suppose $M\cong_\qq \s^k$ with $k>1$ odd. Again we may assume that both $F$ and $B$ are rationally elliptic. Since the Leray--Serre spectral sequence does not have to degenerate at the $E_2$-term in this case, we argue as follows: The contractible algebra $(C,\tilde \dif_0)$ from \eqref{eqn14} is of the form $\Lambda (\tilde C \oplus \tilde \dif_0 \tilde C)$ for a graded vector space $\tilde C$. The vector space $\tilde C$ is finite dimensional, as it a subspace of $V\oplus W$. Indeed, the proof of \cite[Theorem 14.9, p.~187]{FHT01} shows that the decomposition \eqref{eqn14} may be obtained by merely changing the differential. It follows that $(\Lambda \tilde C,\tilde \dif_0)$ is an elliptic two-stage Sullivan algebra. We may now form the associated pure algebra $(\Lambda \tilde C,(\tilde \dif_0)_\sigma)$ (cf.~\cite[p.~438]{FHT01}). By \cite[Theorem 32.4, p.~438]{FHT01} the associated pure algebra is also elliptic. This implies that
\begin{align*}
\tilde C^\even=\tilde\dif_0(\tilde C)
\end{align*}
As a consequence of the proof of \cite[Theorem 14.9, p.~187]{FHT01} we have $\tilde C^\even\In V$ and $\tilde C^\odd \In W$. In other words, in particular, an odd-degree element of $V$ passes injectively to the minimal model of $M$, i.e.~
\begin{align*}
{\pi_\odd(p)\otimes \qq}&\co\pi_\odd(M)\otimes \qq\twoheadrightarrow\pi_\odd(B)\otimes \qq
\intertext{is surjective. Vice versa, the same arguments apply to show that}
{\pi_\even(i)\otimes \qq}&\co\pi_\even(F)\otimes \qq\hto{}\pi_\even(M)\otimes \qq
\end{align*}
is injective.

Since $\dim \pi_\even(M)=0$, we have $W^\even=0$.
Since $B$ is rationally elliptic and since $\tilde C^\even\In V$, we conclude that $\dim V^\odd\geq \dim \tilde C^\even$.
However, since $\pi_\odd(M)=\pi_*(\s^k)=\qq [\s^k]$, we obtain that
\begin{align}\label{eqn15}
\dim \tilde C^\even \leq \dim V^\odd\leq 1
\end{align}
We have $\tilde C^\odd=\tilde C^{\even-1}$.

All this amounts to discerning three cases:
\begin{enumerate}
\item $\tilde C^\odd=\tilde C^l\cong \qq\cong \tilde C^{l+1}=\tilde C^\even$ (for some odd $l>1$), $V^\odd=V^k=\qq$, $V^\even=\tilde C^\even$, $W=\tilde C^\odd$.
\item $\tilde C=0$, $V=V^k=\qq$, $W=0$.
\item $\tilde C=0$, $W=W^k=\qq$, $V=0$.
\end{enumerate}
(Note that a case analogous to case one of the form ``$\tilde C^\odd=\tilde C^l\cong \qq\cong\tilde C^{l+1}=\tilde C^\even$ (for some odd $l>1$), $V^\odd=0$, $V^\even=\tilde C^\even$, $W=\tilde C^\odd\oplus \qq[\s^k]$'' cannot occur due to relation \eqref{eqn15}.)

In cases (2) and (3) the fibration is trivial. In case one we compute
\begin{align*}
\dim B-\dim F&=k>0.
\end{align*}
\end{prf}
This theorem clearly provides one particular answer to Question \ref{ques01}.

The next corollary confirms the Wilhelm conjecture under stronger geometric assumptions.
\begin{cor}
Let $M^n$ be a simply-connected closed Riemannian manifold with either
\begin{itemize}
\item $2$-non-negative curvature operator (and positive curvature operator, in particular), or
\item weakly quarter-pinched curvature, or
\item
$\operatorname{sec} \geq 1$ and $\operatorname{diam} M \geq \pi/2$,
\end{itemize}
then $M$ satisfies the Wilhelm conjecture.
\end{cor}
\begin{prf}
In all three cases the respective results in \cite{BW08},  \cite{Ber60}, \cite{GG87} and \cite{Wil01} imply that $M$ is a compact rank one symmetric space and therefore has singly generated cohomology.
\end{prf}
\begin{rem}
A classical example of the type of fibration we are dealing with in case one of the proof is the twistor bundle $\s^2\hto{}\cc\pp^{2n+1}\to \hh\pp^n$. An example of a non-trivial bundle in case two is $\s^3\hto{} \s^{4n+3} \to \hh\pp^n$.
\end{rem}

\begin{theo}\label{theo06}
The Wilhelm conjecture holds true on all the known even-dimensional examples of manifolds admitting positive curvature.
\end{theo}
\begin{prf}
The CROSSes are covered by Theorem \ref{theo05}. According to \cite{Zil07} it remains to consider
\begin{enumerate}
\item the flag manifolds $W^6=\SU(3)/T^2, W^{12}=\Sp(3)/\Sp(1)^3, W^{24}=\F_4/\Spin(8)$,
\item $\biq{\SU(3)}{T^2}$.
\end{enumerate}
We only give the arguments for $W^6$, the result for the remaining spaces follows completely analogously up to a degree shift. The space $W^6$ fits into the rationally non-trivial fibration
\begin{align*}
\s^2\to W^6\to \cc\pp^2
\end{align*}
We prove the Wilhelm conjecture for the total space of every such non-trivial rational fibration. This will also prove $(2)$. For the other homogeneous spaces there are corresponding bundles (cf.~\cite[p.~17]{Zil07}). Indeed, by \cite[p.~51]{Esc84} (and the formality of positively elliptic manifolds), the minimal model of $W^6$ is given by
\begin{align*}
(\Lambda \langle a,b,x,y\rangle, a\mapsto 0, b\mapsto 0, x\mapsto a^2+ab+b^2, y\mapsto b^3)
\intertext{the one of $\biq{\SU(3)}{T^2}$ by}
(\Lambda \langle a,b,x,y\rangle, a\mapsto 0, b\mapsto 0, x\mapsto a^2+ab-b^2, y\mapsto b^3)
\end{align*}
with $\deg a=\deg b=2$. The bundles for $W^{12}$ and $W^{24}$ are rationally non-trivial as well. This can be seen from the minimal model of $W^{12}$, which easily computes as
\begin{align*}
(\Lambda \langle a,b,x,y\rangle, a\mapsto 0, b\mapsto 0, x\mapsto a^2+ab+b^2, y\mapsto b^3)
\end{align*}
for $\deg a=\deg b=4$. The minimal model of $W^{24}$ can be derived from \cite[Theorem 1.1]{MW}, where the integral cohomology algebra of $W^{24}$ is given. It computes as
\begin{align*}
(\Lambda \langle a,b,x,y\rangle, a\mapsto 0, b\mapsto 0, x\mapsto a^2-ab+b^2, y\mapsto b^3)
\end{align*}
with $\deg a=\deg b=8$.

(For this we denote the generators of the integral cohomology $(2e(\mathcal{E}_1)+e(\mathcal{E}_2))/3$ by $a$ and $(e(\mathcal{E}_1)+2e(\mathcal{E}_2))/3$ by $b$ and the relations are then given as the second and third elementary symmetric polynomials in $(a,b-a,-b)$---using the terminology from \cite[Theorem 1.1]{MW}.)

\vspace{3mm}

So let us show that the total space $E$ of a rationally non-trivial bundle from above---for the sake of simplicity we work with an $\s^2$-bundle over $\cc\pp^2$---satisfies Question \ref{ques01}, i.e.~given a fibration $F\hto{}E\to B$ as described, then $\dim B>\dim F$. Again we derive that both fiber and base space are rationally elliptic of positive Euler characteristic. Thus the rational cohomology module of the total space splits as a product of the ones of fiber and base, since the Leray--Serre spectral sequence degenerates at the $E_2$-term for degree reasons. Consequently, the generators of the cohomology algebra of $F$ map injectively into the cohomology algebra of $E$.

Given the concrete minimal models, we compute easily that there is no element $[z]\in H^2(E)$ with $[z]^2=0$. (The same holds true in all the other cases for elements of degree $2,4,8$ respectively.) So suppose in a fibration $F\hto{} E\to B$ the cohomology of $B$ is not trivial. We may assume it is not identical to the cohomology of $E$, since the fibration would be trivial in this case. Analogously, we may suppose that the cohomology of $E$ does not equal the one of $F$. There remain two cases to be considered.

\case{1}
Assume that $H^*(B)$ is not entirely generated by elements of the contractible algebra $C$ from the splitting given in \eqref{eqn14}. That is, there is a closed element $z\in V$---again we denote by $(\Lambda V,\dif)$ the minimal model of $B$ and by $(\Lambda W,\bar\dif)$ the one of $F$---which represents a non-trivial cohomology class in $H^*(E)$. Then $[z]^2\neq 0$, i.e.~the subalgebra generated by $[z]$ has formal dimension $4$ at least, i.e.~$\dim B\geq 4$ and $\dim F\leq 2$.

\case{2}
Assume that $H^*(B)$ is generated as an algebra by (even-degree) elements of $C\neq 0$. This implies, in particular, that $H^2(E)\to H^2(F)$ is an isomorphism. Consequently, the lowest degree element of $V$ has degree at least $4$ and $\dim B\geq 4>\dim F$.

In the cases $W^{12}$ and $W^{24}$ the lowest degree non-trivial element of $B$ has to be in degrees larger or equal to $8$ respectively $16$.
\end{prf}

\begin{cor}
Let $M^n$ be an even-dimensional simply-connected closed Riemannian manifold with isometry group satisfying $\dim \Isom(M^n)\geq 2n-6$, then $M$ satisfies the Wilhelm conjecture.
\end{cor}
\begin{prf}
This follows directly from the classification results in  \cite{Wilking06}, \cite{Wallach72}, \cite{Ber76} combined with Theorems \ref{theo05} and \ref{theo06}.
\end{prf}

\begin{rem}\label{rem01}
Let $M^n$, $n$ even, be Hard-Lefschetz, rationally elliptic, positively curved with symmetry rank at least $2\log_2n+n/8$. If $M$ fits into a fibration $F\hto{} M\xto{p} B$ as in Question \ref{ques01}, then $\dim F- \dim B\leq n/2$---a case uniquely realized by a (rational) $\cc\pp^{(n-l)/2}$--fibration over $\s^{l}$.

\vspace{3mm}

This is the situation we are investigating in Corollary \ref{cor02}. Thus we derive that we are either in the case of a complex projective space---here we may quote Theorem \ref{theo05}---or the minimal model of $M$ has the very special structure provided by the corollary.

As in Theorem \ref{theo05} we may assume $F$ and $B$ to be rationally elliptic. First we argue that---unless the fibration is trivial---the Hard-Lefschetz form $[u]$ has to lie in the fibre cohomology. Indeed, if it comes from the base, we compute $0=[u]^{\dim B+1}=(H^2(p)([u]))^{\dim B+1}\neq 0$ unless $\dim B=\dim M$.

\case{1} Assume both the elements $u$ and $a$ lie in the minimal model of the fiber. Let us show that in this case $\dim B=0$ and the fibration is trivial. For this we show that the model of the fibration is minimal already, i.e.~the algebra $(C,\tilde \dif_0)$ from \eqref{eqn14} is trivial.

Since $M$ has positive Euler characteristic, again we derive $\chi(F)>0$, $\chi(B)>0$. This implies that $B$ is positively elliptic. Since the cohomology generators of $M$ already lie in $F$ by assumption, we see that $C$ consists of at most two generators in even degrees and that $(\Lambda V,\dif)$ , the model of the base, satisfies $\dim V^\even\leq 2$. However, it is now easy to check that one cannot specify elements $x',y'\in W^\odd$ (together with a suited minimal model of the base)---where $(\Lambda W,\bar\dif)$ is the minimal model of $F$---with $\bar \dif x'\neq 0,\bar \dif y'\neq 0$ and $(\dif x')|_V\neq 0,(\dif y')|_V\neq 0$ such that the given minimal model of $M$ can result. This is due to the fact that $a^2$ and $au^{(n-2\deg a+2)/2}$ are summands of the respective differentials. That is, both differentials contain the factor $a$ in a summand, however, since $a$ appears in its second power at most, the given relations cannot come from relations of the base (after gluing). The same holds true if there is exactly one generator in even degree in $C$. Thus $C$ has to be trivial.

\case{2} Assume $u$ is in the minimal model of the fiber and $a$ lies in the minimal model of the base. Then the element $x$ must lie in the minimal model of the base---otherwise the base could not be finite-dimensional

The contractible algebra $C$ contains at most one generator of even degree. A direct check shows that there are two different cases. Either the model of the fibration is given by
\begin{align*}
(\Lambda \langle u,x',u',x, a,y\rangle, x'\mapsto u', u\mapsto 0, u'\mapsto 0,\\
x\mapsto a^2+k_1a(u')^{\deg a/\deg u'}+k_2\cdot (u')^{2\deg a/\deg u'}, \\
y\mapsto a(u')^{(n-2\deg a+2)/\deg u'}+k_3(u')^{(n-\deg a+2)/\deg u'})
\intertext{or by}
(\Lambda \langle u,x, a,y\rangle, x\mapsto a^2, y\mapsto u^{(n-\deg a+2)/2}+k\cdot au^{(n-2\deg a+2)/2} )
\end{align*}
with $k\neq 0$ in the second case.
In the first case the model of the base is generated by $u',a,x,y$ and the one of the fiber by $u,x'$. Consequently, we have that $\dim B>\dim F$. In the second case the model is minimal already. The model of the base is generated by $a,x$ and the model of the fibre by $u,y$. This case corresponds to a (rational) $\cc\pp^{(n-\deg a+2)/2}$--fibration over $\s^{\deg a}$. Since $\deg a\geq n/4+1$, we compute $\dim F-\dim B=(n-\deg a+2)-\deg a\leq n+2-2(n/4+1)=n/2$.
\end{rem}

\begin{theo}\label{theo03}
If $M^{2n}$ is a positively curved manifold with the rational homotopy type of a simply-connected hermitian symmetric space $N$ and if $M$ has symmetry rank at least $2\log_2 (2n)+6$, then $N$ is either $\cc\pp^{n}$ or $\SO(n+2)/\U(1)\times \SO(n)$ and gives a positive answer to Question \ref{ques01}.
\end{theo}
\begin{prf}
The classification result is Proposition \ref{prop01}. In the case of a rational $\cc\pp^n$ we may quote Theorem \ref{theo05}. In the case of $\SO(n+2)/\U(1)\times \SO(n)$ we argue as in Remark \ref{rem01}. Indeed, the minimal model of this K\"ahler manifold is covered by the minimal models considered there. Case 1 applies equally. As for Case 2 we need to exclude a possible occurrence of the minimal model
\begin{align*}
(\Lambda \langle u,x, a,y\rangle, x\mapsto a^2, y\mapsto u^{(2n-\deg a+2)/2}+k\cdot au^{(2n-2\deg a+2)/2} )
\end{align*}
For this we compute the minimal model of $\SO(n+2)/\U(1)\times \SO(n)$ as the one of $\cc\pp^n$ if $n$ is odd and as
\begin{align*}
(\Lambda \langle u,a,x,y\rangle, x\mapsto a^2+u^{n}, y\mapsto au)
\end{align*}
with $\deg a=n$ if $n$ is even. Consequently, this space does not fit into a non-trivial rational fibration unless $\dim B>\dim F$.
\end{prf}



\vfill

\begin{center}
\noindent
\begin{minipage}{\linewidth}
\small \noindent \textsc
{Manuel Amann} \\
\textsc{Fakult\"at f\"ur Mathematik}\\
\textsc{Institut f\"ur Algebra und Geometrie}\\
\textsc{Karlsruher Institut f\"ur Technologie}\\
\textsc{Kaiserstra\ss e 89--93}\\
\textsc{76133 Karlsruhe}\\
\textsc{Germany}\\
[1ex]
\textsf{manuel.amann@kit.edu}\\
\textsf{http://topology.math.kit.edu/$21\_54$.php}
\end{minipage}
\end{center}

\vspace{10mm}

\begin{center}
\noindent
\begin{minipage}{\linewidth}
\small \noindent \textsc
{Lee Kennard} \\
\textsc{Department of Mathematics}\\
\textsc{South Hall}\\
\textsc{University of California}\\
\textsc{Santa Barbara, CA 93106-3080}\\
\textsc{USA}\\
[1ex]
\textsf{kennard@math.ucsb.edu}\\
\end{minipage}
\end{center}

\end{document}